\documentclass[a4paper,12pt,reqno]{amsart}
\usepackage{amsmath,amsthm,amssymb}
\usepackage{graphicx, color}
\usepackage[numbers,sort&compress]{natbib}
\nonstopmode \numberwithin{equation}{section}
\newtheorem{definition}{Definition}[section]
\newtheorem{theorem}{Theorem}[section]
\newtheorem{proposition}{Proposition}[section]
 \newtheorem{corollary}{Corollary}[section]

\newtheorem{example}{Example}[section]
\newtheorem{remark}{Remark}[section]

\allowdisplaybreaks

\allowdisplaybreaks

\begin{document}

%%%%%%%%%%%%%%%%%%%%%%%%%%%%%%%%%%%%%%%%%%%%%%%%
\title[On a certain extension of the Riemann-Liouville... ]{On a certain extension of the Riemann-Liouville fractional derivative operator}
%%%%%%%%%%%%%%%%%%%%%%%%%%%%%%%%%%%%%%%%%%%%%%%%%%%%%%%%%%

\author[ K.S. Nisar, G. Rahman, Z. Tomovski ]{ Kottakkaran Sooppy  Nisar,  Gauhar Rahman,  Zivorad Tomovski}

\address{Kottakkaran Sooppy  Nisar:    Department of Mathematics, College of Arts and Science-Wadi Aldawaser, 11991,
Prince Sattam bin Abdulaziz University, Alkharj, Kingdom of Saudi Arabia}
\email{n.sooppy@psau.edu.sa; ksnisar1@gmail.com}

\address{Gauhar Rahman:    Department of Mathematics, International Islamic  University, Islamabad, Pakistan}
\email{gauhar55uom@gmail.com}

\address{Zivorad Tomovski: University ”St. Cyril and Methodius”, Faculty of Natural Sciences and Mathematics, Institute of Mathematics, Repubic of Macedonia.}
\email{tomovski@pmf.ukim.edu.mk}



%\thanks{*Corresponding author}

\begin{abstract}
 The main aim of this present paper is to present a new  extension of the fractional derivative operator by using the extension of Beta function recently defined by Shadab et al. \cite{Choi2017}. Moreover, we establish some results related to the newly defined modified fractional derivative operator such as Mellin transform and relations to extended hypergeometric and Appell's function via generating functions.
 \end{abstract}

%%%%%%%%%%%%%%%%%%%%%%%%%%%%%%%%%%%%%%%%%%%%%%%%%%%%%%%%%%%%%%%%%%%%%%%%%%%%%%

\maketitle

\section{Introduction and preliminaries}\label{intro}

We begin with the well-known Riemann-Liouville (R-L) fractional derivative of order $\mu$ is defined  (see \cite{Samko1933},\cite{Kilbas}) by
\begin{eqnarray}\label{eq1}
\mathfrak{D}_{x}^{\mu}\{f(x)\}=\frac{1}{\Gamma(-\mu)}\int_0^xf(t)(x-t)^{-\mu-1}dt, \Re(\mu)<0.
\end{eqnarray}
For the case $m-1<\Re(\mu)<m$ where $m=1,2,\cdots$, it follows
\begin{align}\label{eq2}
\mathfrak{D}_{x}^{\mu}\{f(x)\}&=\frac{d^m}{dx^m}\mathfrak{D}_{x}^{\mu-m}\Big\{f(x)\Big\}\notag\\
&=\frac{d^m}{dx^m}\Big\{\frac{1}{\Gamma(-\mu+m)}\int_0^xf(t)(x-t)^{-\mu+m-1}dt\Big\}
\end{align}
and
\begin{eqnarray}\label{eq3}
\mathfrak{D}_{x}^{\mu}\{x^\sigma\}=\frac{\Gamma(\sigma+1)}{\Gamma(\sigma-\mu+1)}x^{\sigma-\mu}.
\end{eqnarray}
The researchers (see \cite{Kiymaz,Luo,Olver,Parmar,Srivastava2012}) investigated the various extensions and generalization of fractional derivative operators.

The extended R-L fractional derivative of order $\mu$ is defined in \cite{Ozerslan} by
\begin{eqnarray}\label{Ofrac}
\mathfrak{D}_{x}^{\mu}\{f(x);p\}=\frac{1}{\Gamma(-\mu)}\int_0^xf(t)(x-t)^{-\mu-1}\exp\Big(-\frac{px^2}{t(x-t)}\Big) dt, \Re(\mu)<0.
\end{eqnarray}
For the case $m-1<\Re(\mu)<m$ where $m=1,2,\cdots$, it follows
\begin{align}\label{Ofrac1}
\mathfrak{D}_{x}^{\mu}\{f(x);p\}&=\frac{d^m}{dx^m}\mathfrak{D}_{x}^{\mu-m}\Big\{f(x);p\Big\}\notag\\
&=\frac{d^m}{dx^m}\Big\{\frac{1}{\Gamma(-\mu+m)}\int_0^xf(t)(x-t)^{-\mu+m-1}\exp\Big(-\frac{px^2}{t(x-t)}\Big) dt\Big\}.
\end{align}
An extension of fractional derivative operator established in \cite{Baleanu} is given by
\begin{eqnarray}\label{Efrac}
\mathfrak{D}_{x}^{\mu}\{f(x);p,q\}=\frac{1}{\Gamma(-\mu)}\int_0^xf(t)(x-t)^{-\mu-1}.
\exp\Big(-\frac{px}{t}-\frac{qx}{(x-t)}\Big) dt, \Re(\mu)<0.
\end{eqnarray}
For example,
\begin{align*}
\mathfrak{D}_{x}^{\mu}\left\{x^\nu;p,q\right\}_{x=1}=\frac{B_{p,q}(\nu+1,-\mu)}{\Gamma(-\mu)},
\end{align*}
where $B_{p,q}(x,y)$ is the extended Beta function (see \cite{Mehrez}) defined by
\begin{align*}
B_{p,q}(x,y)=\int_{0}^{1}t^{x-1}(1-t)^{y-1}{\mathcal{E}}_{p,q}(t)dt, x,y,p,q\in \mathbb{C}, \Re(p),\Re(q)>0,
\end{align*}
where ${\mathcal{E}}_{p,q}(t)=e^{-\frac{p}{t}-\frac{q}{1-t}}$.
For $p=q$ we denote $B_{p,q}$ by $B_{p}$ and for $p=q=0$, we get classical Beta function defined by
\begin{eqnarray}\label{beta}
B(x,y)=\int\limits_{0}^{1}t^{x-1}(1-t)^{y-1}dt, (\Re(x)>0, \Re(y)>0).
\end{eqnarray}

For the case $m-1<\Re(\mu)<m$ where $m=1,2,\cdots$, it follows
\begin{align}
\mathfrak{D}_{x}^{\mu}\{f(x);p,q\}&=\frac{d^m}{dx^m}\mathfrak{D}_{x}^{\mu-m}\Big\{f(x);p,q\Big\}\notag\\
&=\frac{d^m}{dx^m}\Big\{\frac{1}{\Gamma(-\mu+m)}\int_0^xf(t)(x-t)^{-\mu+m-1}\notag\\
&\times\exp\Big(-\frac{px}{t}-\frac{qx}{(x-t)}\Big) dt\Big\}.
\end{align}
Recently, Rahman et al. \cite{Rahman} define an extension of extended R-L fractional derivative of order $\mu$ as
\begin{align}\label{FEfrac}
\mathfrak{D}_{x}^{\mu}\{f(x);p,q,\lambda,\rho\}&=\frac{1}{\Gamma(-\mu)}\int_0^xf(t)(x-t)^{-\mu-1}\notag\\
&\times_1F_1\Big[\lambda;\rho;-\frac{px}{t}\Big]\,_1F_1\Big[\lambda;\rho;-\frac{qx}{(x-t)}\Big] dt, \Re(\mu)<0.
\end{align}
For the case $m-1<\Re(\mu)<m$ where $m=1,2,\cdots$, it follows
\begin{align}
&\mathfrak{D}_{x}^{\mu}\{f(x);p,q,\lambda,\rho\}=\frac{d^m}{dx^m}\mathfrak{D}_{x}^{\mu-m}\Big\{f(x);p,q,\lambda,\rho\Big\}\notag\\
&=\frac{d^m}{dx^m}\Big\{\frac{1}{\Gamma(-\mu+m)}\int_0^xf(t)(x-t)^{-\mu+m-1}
\,_1F_1\Big[\lambda;\rho;-\frac{px}{t}\Big]\,_1F_1\Big[\lambda;\rho;-\frac{qx}{(x-t)}\Big] dt\Big\},
\end{align}
where  $\Re(p)>0$ and $\Re(q)>0$. It is clear that when $\lambda=\rho$, then (\ref{FEfrac}) reduce to (\ref{Efrac}).\\

The Gauss hypergeometric  function which is defined (see \cite{Rainville1960}) as
\begin{eqnarray}\label{Hyper}
_2F_1(\sigma_1,\sigma_2;\sigma_3;z)=\sum\limits_{n=0}^{\infty}\frac{(\sigma_1)_n(\sigma_2)_n}{(\sigma_3)_n}\frac{z^n}{n!}, (|z|<1),
\end{eqnarray}
 $\Big(\sigma_1, \sigma_2, \sigma_3\in\mathbb{C}$ and $\sigma_3\neq0,-1,-2,-3,\cdots\Big)$.
 The integral representation of hypergeometric  hypergeometric function is defined by
\begin{eqnarray}\label{Ihyper}
_2F_1(\sigma_1,\sigma_2;\sigma_3;z)=\frac{\Gamma(\sigma_3)}{\Gamma(\sigma_2)\Gamma(\sigma_3-\sigma_2)}
\int_0^1t^{\sigma_2-1}(1-t)^{\sigma_3-\sigma_2-1}(1-zt)^{-\sigma_1}dt,
\end{eqnarray}
$\Big(\Re(\sigma_3)>\Re(\sigma_2)>0, |\arg(1-z)|<\pi\Big)$.\\
The Appell series or bivariate hypergeometric series and its integral representation  is respectively defined by
\begin{eqnarray}\label{CAppell}
F_{1}(\sigma_1,\sigma_2,\sigma_3;\sigma_4;x, y)=\sum\limits_{m,n=0}^{\infty}\frac{(\sigma_1)_{m+n}(\sigma_2)_{m}(\sigma_3)_{n}x^{m}y^{n}}{(\sigma_4)_{m+n}m!n!};
\end{eqnarray}
 for all $ \sigma_1,\sigma_2,\sigma_3,\sigma_4\in \mathbb{C},  \sigma_4\neq 0,-1,-2,-3,\cdots, \quad |x|<1, |y|<1$.\\
\begin{align}\label{FIAppell}
F_{1}\Big(\sigma_1,\sigma_2, \sigma_3,\sigma_4;x,y\Big)&=\frac{\Gamma(\sigma_4)}{\Gamma(\sigma_1)\Gamma(\sigma_4-\sigma_1)}\notag\\
&\times\int\limits_{0}^{1}t^{\sigma_1-1}(1-t)^{\sigma_4-\sigma_1-1}(1-xt)^{-\sigma_2}(1-yt)^{-\sigma_3}dt
\end{align}
$\Re(\sigma_4)>\Re(\sigma_1)>0$, $|\arg(1-x)|<\pi$ and $|\arg(1-y)|<\pi$.\\

Chaudhry et al. \cite{Chaudhry1997} introduced the  extended Beta function is defined by
\begin{eqnarray}\label{Ebeta}
B(\sigma_1,\sigma_2;p)=B_p(\sigma_1,\sigma_2)=\int\limits_{0}^{1}t^{\sigma_1-1}(1-t)^{\sigma_2-1}
e^{-\frac{p}{t(1-t)}}dt
\end{eqnarray}
(where $\Re(p)>0, \Re(\sigma_1)>0, \Re(\sigma_2)>0$) respectively. When $p=0$, then $B(\sigma_1,\sigma_2;0)=B(\sigma_1,\sigma_2)$.\\
The extended  hypergeometric  function introduced in \cite{Chaudhrya2004} by using the definition of extended Beta function $B_p(\delta_1,\delta_2)$ as follows:
\begin{eqnarray}\label{Ehyper}
F_p(\sigma_1,\sigma_2;\sigma_3;z)=\sum\limits_{n=0}^{\infty}\frac{B_p(\sigma_2+n, \sigma_3-\sigma_2)}{B(\sigma_2, \sigma_3-\sigma_2)}(\sigma_1)_n\frac{z^n}{n!},
\end{eqnarray}
where $p\geq0$ and $\Re(\sigma_3)>\Re(\sigma_2)>0$, $|z|<1$.\\
In the same paper, they defined the following integral representations of extended hypergeometric and confluent hypergeometric functions as
\begin{align}\label{IEhyper}
F_p(\sigma_1,\sigma_2;\sigma_3;z)&=\frac{1}{B(\sigma_2, \sigma_3-\sigma_2)}\notag\\
&\times\int_0^1t^{\sigma_2-1}(1-t)^{\sigma_3-\sigma_2-1}(1-zt)^{-\sigma_1}\exp\Big(\frac{-p}{t(1-t)}\Big)dt,
\end{align}
$$\Big(p\geq0, \Re(\sigma_3)>\Re(\sigma_2)>0, |\arg(1-z)|<\pi\Big).$$
The extended Appell's function is defined by (see \cite{Ozerslan})
\begin{eqnarray}\label{EAppell}
F_1(\sigma_1,\sigma_2,\sigma_3;\sigma_4;x,y;p)=\sum\limits_{m,n=0}^{\infty}\frac{B_p(\sigma_1+m+n, \sigma_4-\sigma_1)}{B(\sigma_1, \sigma_4-\sigma_1)}(\sigma_2)_m(\sigma_3)_n\frac{x^my^n}{m!n!}
\end{eqnarray}
where $p\geq0$ and $\Re(\sigma_4)>\Re(\sigma_1)>0$ and $|x|,|y|<1$.\\
 \"{O}zerslan and  \"{O}zergin \cite{Ozerslan} defined its integral representation by
\begin{align}\label{IEAppell}
F_1(\sigma_1,\sigma_2,\sigma_3;\sigma_4;z;p)&=\frac{1}{B(\sigma_1, \sigma_4-\sigma_1)}\notag\\
&\times\int_0^1t^{\sigma_1-1}(1-t)^{\sigma_4-\sigma_1-1}(1-xt)^{-\sigma_2}(1-yt)^{-\sigma_3}
\exp\Big(\frac{-p}{t(1-t)}\Big)dt,\\
&\Big(\Re(p)>0, \Re(\sigma_4)>\Re(\sigma_1)>0, |\arg(1-x)|<\pi, |\arg(1-y)|<\pi\Big)\notag.
\end{align}

It is clear that when $p=0$, then the equations (\ref{Ehyper})-(\ref{IEAppell}) reduce to the well known hypergeometric, confluent hypergeometric and Appell's series and their integral representation respectively.\\
Very recently Shadab et al.  \cite{Choi2017} introduced a new and modified extension of Beta function as:
\begin{eqnarray}\label{Cbeta}
B^{\alpha}_{p}(\sigma_1,\sigma_2)=\int_0^1t^{\sigma_1-1}(1-t)^{\sigma_2-1}E_{\alpha}\Big(-\frac{p}{t(1-t)}\Big)dt,
\end{eqnarray}
where $\Re(\sigma_1)>0$, $\Re(\sigma_2)>0$ and $E_{\alpha}\Big(.\Big)$ is the Mittag-Leffler function defined by
\begin{eqnarray}
E_{\alpha}\Big(z\Big)=\sum_{n=0}^\infty\frac{z^n}{\Gamma(\alpha n+1)}.
\end{eqnarray}
Obviously, when $\alpha=1$  then $B^{1}_{p}(x,y)=B_p(x,y)$ is the extended Beta function (see\cite{Chaudhry1997}). Similarly, when when $\alpha=1$ and $p=0$,  then $B^{1}_{0}(x,y)=B_0(x,y)$ is the classical Beta function.\\
Shadab et al.  \cite{Choi2017} also defined extended hypergeometric function and its integral representation
\begin{align}\label{pqhyper}
F_{p}^{\alpha}(\sigma_1,\sigma_2;\sigma_3;z)={}_2F_{1}\Big(\sigma_1,\sigma_2;\sigma_3;z;p,\alpha\Big)\notag&=
\sum_{n=0}^\infty(\sigma_1)_n\frac{B_{p}^{\alpha}(\sigma_2+n,\sigma_3-\sigma_2)}{B(\sigma_2,\sigma_3-\sigma_2)}
\frac{z^n}{n!}\notag\\
=&\sum_{n=0}^\infty(\sigma_1)_n\frac{B(\sigma_2+n,\sigma_3-\sigma_2; p,\alpha)}{B(\sigma_2,\sigma_3-\sigma_2)}
\frac{z^n}{n!}
\end{align}
where $p,\alpha\geq0$, $\sigma_1,\sigma_2,\sigma_3\in\mathbb{C}$  and $|z|<1$.
 \begin{align}\label{pqIhyper}
F_{p}^{\alpha}(\sigma_1,\sigma_2;\sigma_3;z)&=\frac{1}{\beta(\sigma_2;\sigma_3-\sigma_2)}\notag\\
&\times\int_0^1t^{\sigma_2-1}
(1-t)^{\sigma_3-\sigma_2-1}(1-tz)^{-\sigma_1}E_{\alpha}\Big(-\frac{p}{t(1-t)}\Big)dt,
\end{align}
where $\Re(p)>0$, $\Re(\alpha)>0$,  $\Re(\sigma_3)>\Re(\sigma_2)>0$. Obviously when $\alpha=1$, then the hypergeometric function (\ref{pqhyper}) will reduce to the extended hypergeometric function (\ref{Ehyper}) and similarly when $\alpha=1$ and $p=0$ then the hypergeometric function (\ref{pqhyper}) will reduce to the  hypergeometric function (\ref{Hyper}).\\

For various extensions and generalizations of Beta function and hypergeometric functions the interested readers may refer to the recent work of researchers (see e. g., \cite{Choi2014,Mubeen2017,Ozerslan,Ozergin}).
%%%%%%%%%%%%%%%%%%%%%%%%%%%%%%%%%%%%%%%%%%%%%%%%%%%%%%%%%%%%%%%%%%%%%%%%%%%%%
\section{ extension of Appell's functions and its integral representations}
%%%%%%%%%%%%%%%%%%%%%%%%%%%%%%%%%%%%%%%%%%%%%%%%%%%%%%%%%%%%%%%%%%%%%%%%%%%%%

We start the section by deriving the relation of \eqref{Cbeta} with multi-index Mittag-Leffler function \cite{Kir1} as follows:

\begin{proposition}
For $\Re(p), \Re(\sigma_1), \Re(\sigma_2), \Re(\alpha) > 0$ the following relation holds true:
\begin{align}
B^{\alpha}_{p}(\sigma_1,\sigma_2)=\frac{\pi sin\pi(\sigma_1+\sigma_2)}{(sin\pi\sigma_1)(sin\pi\sigma_2)}E_{(\alpha,1), (1, 1-\sigma_1), (1,1-\sigma_2)}^{(2,1-\sigma_1-\sigma_2)}(-p)
\end{align}
where $E_{(\alpha,1), (1, 1-\sigma_1), (1,1-\sigma_2)}^{(2,1-\sigma_1-\sigma_2)}(-p)$ is the multi-index Mittag-Leffler function \cite{Kir1}.
\end{proposition}

{\bf Proof}
Using the definition of Beta function and reduction theorem of Gamma function, we get
\begin{align*}
B^{\alpha}_{p}(\sigma_1,\sigma_2)&=\sum_{n=0}^{\infty}\frac{(-p)^n}{\Gamma(\alpha n+1)}\int_{0}^{1}t^{\sigma_1-1}(1-t)^{\sigma_2-1}\frac{1}{t^n(1-t)^n}dt\\
&=\sum_{n=0}^{\infty}\frac{(-p)^n}{\Gamma(\alpha n+1)}\int_{0}^{1}t^{\sigma_1-n-1}(1-t)^{\sigma_2-n-1}dt\\
&=\sum_{n=0}^{\infty}\frac{(-p)^n}{\Gamma(\alpha n+1)}B(\sigma_1-n, \sigma_2-n)\\
&=\sum_{n=0}^{\infty}\frac{(-p)^n}{\Gamma(\alpha n+1)}\frac{\Gamma(\sigma_1-n)\Gamma(\sigma_2-n)}{\Gamma(\sigma_1+\sigma_2-2n)}\\
&=\frac{\Gamma(-\sigma_1)\Gamma(1+\sigma_1)\Gamma(-\sigma_2)\Gamma(1+\sigma_2)}{\Gamma(-\sigma_1-\sigma_2)\Gamma(1+\sigma_1+\sigma_2)}\\
&\times \sum_{n=0}^{\infty} \frac{\Gamma(2n+1-\sigma_1-\sigma_2)(-p)^n}{\Gamma(\alpha n+1)\Gamma(n+1-\sigma_1)\Gamma(n+1-\sigma_2)}
\end{align*}
Now, using the Euler's reflection formula on Gamma function,
\begin{align}\label{rel1}
{\Gamma(r)\Gamma(1-r)}=\frac{\pi}{\sin(\pi r)},
\end{align}
we get the desired result.

Next, we used the definition (\ref{Cbeta}) and consider the following modified extension Appell's  functions.
\begin{definition}\label{def2}
The modified extended Appell's function $F_1$ is defined by
\begin{align}\label{Appell}
F_{1,p}^{\alpha}(\sigma_1,\sigma_2,\sigma_3;\sigma_4;x,y)&=
F_{1}(\sigma_1,\sigma_2,\sigma_3;\sigma_4;x,y;p,\alpha)\notag\\&=\sum_{m,n=0}^\infty(\sigma_2)_m(\sigma_3)_n
\frac{B_{p}^{\alpha}(\sigma_1+m+n,\sigma_4-\sigma_1)}{B(\sigma_1,\sigma_4-\sigma_1)}
\frac{x^m}{m!}\frac{y^n}{n!}\notag\\
=&\sum_{m,n=0}^\infty(\sigma_2)_m(\sigma_3)_n\frac{B(\sigma_1+m+n,\sigma_4-\sigma_1; p,\alpha)}{B(\sigma_1,\sigma_4-\sigma_1)}
\frac{x^m}{m!}\frac{y^n}{n!}
\end{align}
where $p,\alpha\geq0$, $\sigma_1,\sigma_2,\sigma_3,\sigma_4\in\mathbb{C}$ and  $|x|<1$, $|y|<1$.
\end{definition}
\begin{remark}
Setting  $\alpha=1$ in (\ref{Appell}), then we get the extended Appell's functions (see \cite{Ozerslan}).
\end{remark}
%%%%%%%%%%%%%%%%%%%%%%%%%%%%%%%%%%%%%%%%%%%%%%%%%%%%%%%%%%%%%%%%%%%%%%%%%%%%%%%%%%%%%%%%%%%%%%%%%%%%%%%%%%%%%%%
Now, we derive the following proposition
\begin{proposition} For $p,q >0$, $0<x, y<\frac{1}{2}$, the following inequality holds true:
\begin{align}
B_{p,q}(x,y) \leq (2p)^{\frac{2x-1}{2}}(2q)^{\frac{2y-1}{2}}\sqrt{\Gamma(-2x+1,2p)\Gamma(-2y+1,2q)}.
\end{align}
\end{proposition}
\begin{proof}
Applying the Cauchy-Schwarz integral inequality, we obtain:
\begin{align*}
B_{p,q}(x,y)&=\int_{0}^{1}t^{x-1}(1-t)^{y-1}e^{-\frac{p}{t}}e^{-\frac{q}{t}} dt\\
&\leq	\left(\int_{0}^{1}(t^{x-1}e^{-p/t})^2dt\right)^{1/2}\left(\int_{0}^{1}[(1-t)^{y-1}e^{-q/1-t}]^2dt\right)^{1/2}\\
&=\left(\int_{1}^{\infty}t^{-2x}e^{-2 pt}dt\right)^{1/2}\left(\int_{1}^{\infty}(1-t)^{-2y}e^{-2 qt}dt\right)^{1/2}\\
&=\left[(2p)^{2x-1}\Gamma(-2x+1, 2p)\right]^{1/2}\left[(2q)^{2y-1}\Gamma(-2y+1, 2q)\right]^{1/2}\\
&=(2p)^{\frac{2x-1}{2}}(2q)^{\frac{2y-1}{2}}\sqrt{\Gamma(-2x+1,2p)\Gamma(-2y+1,2q)},
\end{align*}
where $\Gamma (x,y)$ is the incomplete Gamma function.
\end{proof}
%%%%%%%%%%%%%%%%%%%%%%%%%%%%%%%%%%%%%%%%%%%%%%%%%%%%%%%%%%%%%%%%%%%%%%%%%%%%%%%%%%%%%%%%%%%%%%%%%%%%%%%%%%%%%%%%

\begin{theorem}\label{th2}
The following integral representation holds true for (\ref{Appell})
\begin{align}\label{pqIhyper2}
F_1\Big(\sigma_1,\sigma_2,\sigma_3;\sigma_4;x,y;p,\alpha\Big)&=\frac{1}{B(\sigma_1;\sigma_4-\sigma_1)}
\int_0^1t^{\sigma_1-1}(1-t)^{\sigma_4-\sigma_1-1}
(1-tx)^{-\sigma_2}(1-ty)^{-\sigma_3}\notag\\&\times
E_{\alpha}\Big(-\frac{p}{t(1-t)}\Big)dt,
\end{align}
where $\Re(p)>0$, $\Re(\alpha)>0$,  $\Re(\sigma_4)>\Re(\sigma_1)>0$.
\end{theorem}
\begin{proof}
Using the definition (\ref{Cbeta}) in (\ref{Appell}), we have
\begin{align}\label{PS2}
F_1\Big(\sigma_1,\sigma_2,\sigma_3;\sigma_4;x,y;p,\alpha\Big)&=\frac{1}{B(\sigma_1;\sigma_4-\sigma_1)}
\int_0^1t^{\sigma_1-1}(1-t)^{\sigma_4-\sigma_1-1}\notag\\
\times&\Big(\sum_{m,n=0}^\infty\frac{(\sigma_2)_m(\sigma_3)_n(tx)^m(ty)^n}{m!n!}\Big) dt.
\end{align}
 Since
\begin{align}\label{PS}
\sum_{m,n=0}^\infty\frac{(\sigma_2)_m(\sigma_3)_n(x)^m(y)^n}{m!n!}=(1-tx)^{-\sigma_2}(1-ty)^{-\sigma_3}.
\end{align}

Thus by using (\ref{PS}) in (\ref{PS2}),   we get the desired result.
\end{proof}
%%%%%%%%%%%%%%%%%%%%%%%%%%%%%%%%%%%%%%%%%%%%%%%%%%%%%
\section{Extension of fractional derivative operator}
%%%%%%%%%%%%%%%%%%%%%%%%%%%%%%%%%%%%%%%%%%%%%%%%%%%%%
In this section, we define a new and modified extension of   Riemann-Liouville fractional derivative and obtain its related results.
\begin{definition}\label{frac}
\begin{eqnarray}\label{eq7}
\mathfrak{D}_{z;p}^{\mu;\alpha}\{f(z)\}=\frac{1}{\Gamma(-\mu)}\int_0^zf(t)(z-t)^{-\mu-1}E_\alpha\Big(-\frac{pz^2}{t(z-t)}\Big)dt, \Re(\mu)<0.
\end{eqnarray}
For the case $m-1<\Re(\mu)<m$ where $m=1,2,\cdots$, it follows
\begin{align}\label{eq8}
\mathfrak{D}_{z;p}^{\mu;\alpha}\{f(z)\}&=\frac{d^m}{dz^m}\mathfrak{D}_{z;p}^{\mu-m;\alpha}\Big\{f(z)\Big\}\notag\\
=&\frac{d^m}{dz^m}\Big\{\frac{1}{\Gamma(-\mu+m)}\int_0^zf(t)(z-t)^{-\mu+m-1}E_\alpha\Big(-\frac{pz^2}{t(z-t)}\Big)dt\Big\}.
\end{align}
\end{definition}
\begin{remark}
Obviously if  $\alpha=1$, then (\ref{eq7}) and (\ref{eq8}) respectively reduces to the extended fractional derivative (\ref{Ofrac}) and (\ref{Ofrac1}). Similarly, if we set $\alpha=1$ and $p=0$ we get (\ref{eq1}) and (\ref{eq2}).
\end{remark}
Now, we prove some theorems involving the modified extension of fractional derivative operator.
\begin{theorem} \label{th1} The following formula hold true,
\begin{eqnarray}\label{fd1}
\mathfrak{D}_{z}^{\mu}\{z^{\eta};p,\alpha\}=\frac{B_{p}^{\alpha}(\eta+1,-\mu)}{\Gamma(-\mu)}z^{\eta-\mu},
\Re(\mu)<0.
\end{eqnarray}
\end{theorem}
\begin{proof}
From (\ref{frac}), we have
\begin{eqnarray}\label{fd2}
\mathfrak{D}_{z}^{\mu}\{z^\eta;p,\alpha\}=\frac{1}{\Gamma(-\mu)}\int_0^zt^\eta(z-t)^{-\mu-1}\,
E_\alpha\Big(-\frac{pz^2}{t(z-t)}\Big) dt.
\end{eqnarray}
Substituting $t=uz$ in (\ref{fd2}), we get
\begin{eqnarray*}
\mathfrak{D}_{z}^{\mu}\{z^\eta;p,\alpha\}&=&\frac{1}{\Gamma(-\mu)}\int_0^1(uz)^\eta(z-uz)^{-\mu-1}\,
E_\alpha\Big(-\frac{pz^2}{uz(z-uz)}\Big) dt\\
&=&\frac{z^{\eta-\mu}}{\Gamma(-\mu)}\int_0^1u^\eta(1-u)^{-\mu-1}\,
E_\alpha\Big(-\frac{p}{u(1-u)}\Big) dt,
\end{eqnarray*}
In view of (\ref{Cbeta}) to the above equation, we get the required result.
\end{proof}
%%%%%%%%%%%%%%%%%%%%%%%%%%%%%%%%%%%%%%%%%%%%%%%%%%%%%%%%%%%%%%%%%%%%%%%%%%%%%%%%%%%%%%%%%%%%%%%%%%%%%%%
\begin{theorem}\label{th2a}
Let $\Re(\mu)>0$ and assume that the function $f(z)$ is analytic at the origin with its Maclaurin expansion given by
$f(z)=\sum_{n=0}^{\infty} a_n z^n$ where $|z|<\delta$ for some $\delta\in \mathbb{R^+}$. Then
\begin{align}
\mathfrak{D}_{z;p}^{\mu;\alpha}\{f(z)\}=\mathfrak{D}_{z}^{\mu}\{f(z);p,\alpha\}&=\sum_{n=0}^\infty a_n\mathfrak{D}_{z}^{\mu}\{z^n;p,\alpha\}\notag\\
=&\frac{1}{\Gamma(-\mu)}\sum_{n=0}^\infty a_nB_p^\alpha(n+1,-\mu)z^{n-\mu}.
\end{align}
\end{theorem}

\begin{proof}
Using the series expansion of the function $f(z)$ in (\ref{eq7}) gives
\begin{eqnarray*}
\mathfrak{D}_{z}^{\mu}\{f(z);p,\alpha\}=\frac{1}{\Gamma(-\mu)}\int_0^z\sum_{n=0}^\infty a_nt^n(z-t)^{-\mu-1}\,
E_\alpha\Big(-\frac{pz^2}{t(z-t)}\Big) dt.
\end{eqnarray*}
The series is uniformly convergent on any closed disk centered at the origin with its radius smaller than $\delta$, so does on the line segment from $0$ to a fixed $z$ for $|z|<\delta$. Thus it guarantee terms by terms integration as follows
\begin{eqnarray*}
\mathfrak{D}_{z}^{\mu}\{f(z);p,\alpha\}&=&\sum_{n=0}^\infty a_n\Big\{\frac{1}{\Gamma(-\mu)}\int_0^z t^n(z-t)^{-\mu-1}\,
E_\alpha\Big(-\frac{pz^2}{uz(z-uz)}\Big) dt\\
&=&\sum_{n=0}^\infty a_n\mathfrak{D}_{z}^{\mu}\{z^n;p,\alpha\}.
\end{eqnarray*}
Now, applying Theorem \ref{th1}, we get
\begin{eqnarray*}
\mathfrak{D}_{z}^{\mu}\{f(z);p,\alpha\}
&=&\frac{1}{\Gamma(-\mu)}\sum_{n=0}^\infty a_nB_p^\alpha(n+1,-\mu)z^{n-\mu}, \Re(\mu)<0.
\end{eqnarray*}
which is the required proof.
\end{proof}
%%%%%%%%%%%%%%%%%%%%%%%%%%%%%%%%%%%%%%%%%%%%%%%%%%%%%%%%%%%%%%%%%%%%
%\begin{theorem}
%The following result holds true:
%\begin{align}
%\mathfrak{D}_{z}^{\mu}\{z^{\eta-1}f(z);p,\alpha\}=\frac{z^{\eta-\mu-1}}{\Gamma(-\mu)}\sum_{n=0}^\infty a_n\,B_p^{\alpha}(\eta+n;-\mu)z^{n}.
%\end{align}
%\end{theorem}
%
%\begin{proof}
%By applying theorem \ref{th2a} and using $f(z)=\sum_{n=0}^{\infty} a_n z^n$, we have
%\begin{align}
%\mathfrak{D}_{z}^{\mu}\{z^{\eta-1}f(z);p,\alpha\}=\sum_{n=0}^\infty a_n\mathfrak{D}_{z}^{\mu}\{z^{n+\eta-1};p,\alpha\}
%\end{align}
%Now applying Theorem \ref{th1}, we get the desired result.
%\end{proof}
%%%%%%%%%%%%%%%%%%%%
\begin{example}
The following result holds true:
\begin{align}
\mathfrak{D}_{z}^{\mu}\{e^z;p,\alpha\}=\frac{z^{-\mu}}{\Gamma(-\mu)}\sum_{n=0}^\infty B_p^\alpha(n+1,-\mu)\frac{z^n}{n!}.
\end{align}
Using the power series of $\exp(z)$ and applying Theorem \ref{th2a}, we have
\begin{align}
\mathfrak{D}_{z}^{\mu}\{e^z;p,\alpha\}=\sum_{n=0}^\infty \frac{1}{n!}\mathfrak{D}_{z}^{\mu}\{z^{n};p,\alpha\}.
\end{align}
Now, applying Theorem \ref{th1}, we get the desired result.
\end{example}
%%%%%%%%%%%%%%%%%%%%
\begin{theorem}\label{th3}
The following formula holds true:
\begin{eqnarray}
\mathfrak{D}_{z}^{\eta-\mu}\{z^{\eta-1}(1-z)^{-\beta};p,\alpha\}=
\frac{\Gamma(\eta)}{\Gamma(\mu)}z^{\mu-1}\,_2F_{1;p}^{\alpha}
\Big(\beta, \eta;\mu;z\Big),
\end{eqnarray}
where $\Re(\mu)> \Re(\eta)>0$ and $|z|<1$.
\end{theorem}
\begin{proof}
By direct calculation, we have
\begin{align*}
\mathfrak{D}_{z}^{\eta-\mu}\{z^{\eta-1}(1-z)^{-\beta};p,\alpha\}
&=\frac{1}{\Gamma(\mu-\eta)}\int_0^zt^{\eta-1}(1-t)^{-\beta}(z-t)^{\mu-\eta-1}E_\alpha\Big(-\frac{pz^2}{t(z-t)}\Big) dt\\
&=\frac{z^{\mu-\eta-1}}{\Gamma(\mu-\eta)}\int_0^zt^{\eta-1}(1-t)^{-\beta}(1-\frac{t}{z})^{\mu-\eta-1}
E_\alpha\Big(-\frac{pz^2}{t(z-t)}\Big) dt.
\end{align*}
Substituting $t=zu$ in the above equation, we get
 \begin{align*}
\mathfrak{D}_{z}^{\eta-\mu}\{z^{\eta-1}(1-z)^{-\beta};p,\alpha\}
&=\frac{z^{\mu-1}}{\Gamma(\mu-\eta)}\int_0^1u^{\eta-1}(1-uz)^{-\beta}(1-u)^{\mu-\eta-1}E_\alpha\Big(-\frac{pz^2}{uz(z-uz)}\Big) du\\
=&\frac{z^{\mu-1}}{\Gamma(\mu-\eta)}\int_0^1u^{\eta-1}(1-uz)^{-\beta}(1-u)^{\mu-\eta-1}E_\alpha\Big(-\frac{p}{u(1-u)}\Big) du
\end{align*}
Using (\ref{pqIhyper}) and after simplification we get the required proof.
\end{proof}
%%%%%%%%%%%%%%%%%%%%%%%%%%%%%%%%%%%%%%%%%%%%%%%%%%%%%%%%%%%%%%%%%%%%%%%%%%%%%%%%%%%%%%%%%%%
\begin{theorem}\label{th4}
The following formula holds true:
\begin{align}\label{fd3}
\mathfrak{D}_{z}^{\eta-\mu}\{z^{\eta-1}(1-az)^{-\alpha}(1-bz)^{-\beta};p,\alpha\}
=\frac{\Gamma(\eta)}{\Gamma(\mu)}z^{\mu-1}F_{1}\Big(\eta,\alpha,\beta;\mu;az,bz;p,\alpha\Big),
\end{align}
where $\Re(\mu)> \Re(\eta)>0$, $\Re(\alpha)>0$, $\Re(\beta)>0$, $|az|<1$ and $|bz|<1$.
\end{theorem}
\begin{proof}
Consider the following power series expansion
\begin{eqnarray*}
(1-az)^{-\alpha}(1-bz)^{-\beta}=\sum_{m=0}^\infty\sum_{n=0}^\infty(\alpha)_m(\beta)_n\frac{(az)^m}{m!}\frac{(bz)^n}{n!}.
\end{eqnarray*}
Now, applying Theorem \ref{th3}, we obtain
\begin{align*}
&\mathfrak{D}_{z}^{\eta-\mu}\{z^{\eta-1}(1-az)^{-\alpha}(1-bz)^{-\beta};p,\alpha\}\\
&=\sum_{m=0}^\infty\sum_{n=0}^\infty(\alpha)_m(\beta)_n\frac{(a)^m}{m!}\frac{(b)^n}{n!}
\mathfrak{D}_{z}^{\eta-\mu}\{z^{\eta+m+n-1};p,\alpha\}.
\end{align*}
Using Theorem \ref{th1}, we have
\begin{align*}
&\mathfrak{D}_{z}^{\eta-\mu}\{z^{\eta-1}(1-az)^{-\alpha}(1-bz)^{-\beta};p,\alpha\}\\
&=\sum_{m=0}^\infty\sum_{n=0}^\infty(\alpha)_m(\beta)_n\frac{(a)^m}{m!}\frac{(b)^n}{n!}
\frac{B_{p}^{\alpha}(\eta+m+n,\mu-\eta)}
{\Gamma(\mu-\eta)}z^{\mu+m+n-1}.
\end{align*}
Now, applying (\ref{Appell}), we get
\begin{align*}
\mathfrak{D}_{z}^{\eta-\mu}\{z^{\eta-1}(1-az)^{-\alpha}(1-bz)^{-\beta};p,\alpha\}
=\frac{\Gamma(\eta)}{\Gamma(\mu)}z^{\mu-1}F_{1}\Big(\eta,\alpha,\beta;\mu;az,bz;p,\alpha\Big).
\end{align*}
\end{proof}
%%%%%%%%%%%%%%%%%%%%%%%%%%%%%%%%%%%%%%%%%%%%%%%%%%%%%%%%%%%%%%%
\begin{theorem}\label{th5}
The following Mellin transform formula holds true:
\begin{eqnarray}\label{Mellin}
M\Big\{\mathfrak{D}_{z;p}^{\mu;\alpha}(z^{\eta});p\rightarrow r,\Big\}=\frac{\pi}{\sin(\pi r)}\frac{z^{\eta-\mu}}{\Gamma(-\mu)\Gamma(1-r\alpha)}B(\eta+r+1,-\mu+r),
\end{eqnarray}
where $\Re(\eta)>-1$, $\Re(\mu)<0$,  $\Re(r)>0$.
%and $\Gamma^{\alpha}(.)$ is the new modified gamma function recently defined by Pucheta \cite{Pucheta} as
%\begin{eqnarray}\label{gam}
%\Gamma^\alpha(x)=\int_0^\infty t^{x-1}E_\alpha(-bv)dt.
%\end{eqnarray}
\end{theorem}
\begin{proof}
Applying the Mellin transform on definition (\ref{eq7}), we have
\begin{align*}
&M\Big\{\mathfrak{D}_{z;p}^{\mu;\alpha}(z^{\eta});p\rightarrow r\Big\}=\int_0^\infty p^{r-1}\mathfrak{D}_{z;p}^{\mu;\alpha}(z^{\eta})dp\\
&=\frac{1}{\Gamma(-\mu)}\int_0^\infty p^{r-1}\Big\{\int_0^z t^\eta(z-t)^{-\mu-1}\,
E_\alpha\Big(-\frac{pz^2}{t(z-t)}\Big)dt\Big\}dp \\
&=\frac{z^{-\mu-1}}{\Gamma(-\mu)}\int_0^\infty p^{r-1}\Big\{\int_0^z t^\eta(1-\frac{t}{z})^{-\mu-1}\,
E_\alpha\Big(-\frac{pz^2}{t(z-t)}\Big)dt\Big\}dp \\
&=\frac{z^{\eta-\mu}}{\Gamma(-\mu)}\int_0^\infty p^{r-1}\Big\{\int_0^1 u^\eta(1-u)^{-\mu-1}\,
E_\alpha\Big(-\frac{p}{u(1-u)}\Big)du\Big\}dp
\end{align*}
From the uniform convergence of the integral, the order of integration can be interchanged. Thus, we have
\begin{align}\label{Th-eqn-Mellin}
M\Big\{\mathfrak{D}_{z;p}^{\mu;\alpha}(z^{\eta});p\rightarrow r\Big\}&=\frac{z^{\eta-\mu}}{\Gamma(-\mu)}\notag\\
&\times\int_0^1 u^\eta(1-u)^{-\mu-1}\Big(\int_0^\infty p^{r-1}\,E_\alpha\Big(-\frac{p}{u(1-u)}\Big)dp\Big)du.
\end{align}
Letting $v=\frac{p}{u(1-u)}$, \eqref{Th-eqn-Mellin} reduces to
\begin{align*}
&M\Big\{\mathfrak{D}_{z;p}^{\mu;\alpha}(z^{\eta});p\rightarrow r\Big\}=\frac{z^{\eta-\mu}}{\Gamma(-\mu)}\int_0^1 u^{\eta+r}(1-u)^{-\mu+r-1}\Big(\int_0^\infty v^{r-1}\,E_\alpha\Big(-v\Big)dv\Big)du.
\end{align*}
By using the following formula,
\begin{align}\label{gam}
\int_{0}^{\infty}v^{r-1}E_{\alpha,\gamma}^{\delta}(-wv)dv=\frac{\Gamma(r)\Gamma(\delta-r)}{\Gamma(\delta)w^{r}\Gamma(\gamma-r\alpha)},
\end{align}
 for $\gamma=\delta=1$ and $w=1$, we have
\begin{align*}
M\Big\{\mathfrak{D}_{z;p}^{\mu;\alpha}(z^{\eta});p\rightarrow r\Big\}&=\frac{z^{\eta-\mu}\Gamma(r)\Gamma(1-r)}{\Gamma(-\mu)\Gamma(1-r\alpha)}\int_0^1 u^{\eta+r}(1-u)^{-\mu+r-1}du\\
&=\frac{z^{\eta-\mu}\Gamma(r)\Gamma(1-r)}{\Gamma(-\mu)\Gamma(1-r\alpha)}B(\eta+r+1,-\mu+r),
\end{align*}
Now, using \eqref{rel1} we get the desired result.
%the Euler's reflection formula on Gamma function,
%\begin{align}\label{ERF}
%{\Gamma(r)\Gamma(1-r)}=\frac{\pi}{\sin(\pi r)},
%\end{align}
\end{proof}
%%%%%%%%%%%%%%%%%%%%%%%%%%%%%%%%%%%%%%%%%%%%
\begin{theorem}\label{th6}
The following Mellin transform formula holds true:
\begin{align}\label{Mellin1}
M\Big\{\mathfrak{D}_{z;p}^{\mu;\alpha}((1-z)^{\alpha});p\rightarrow r\Big\}&=z^{-\mu}\frac{\pi}{\sin(\pi r)}\frac{B(1+r,-\mu+r)}{\Gamma(-\mu)\Gamma(1-r\alpha)}\notag\\
&\times_2F_1\Big(\lambda,r+1;1-\mu+2r;z\Big),
\end{align}
where $\Re(p)>0$, $\Re(\mu)<0$, $\Re(r)>0$.
\end{theorem}
%%%%%%%%%%%%%%%%%%%%%%%%%%%%%%%%%%%%%%%%
\begin{proof}
Applying Theorem \ref{th5} with $\eta=n$, we can write
\begin{align*}
M\Big\{\mathfrak{D}_{z;p}^{\mu;\alpha}((1-z)^{\lambda});p\rightarrow r\Big\}&=
\sum_{n=0}^\infty\frac{(\lambda)_n}{n!}
M\Big\{\mathfrak{D}_{z;p}^{\mu;\alpha}(z^{n});p\rightarrow r\Big\}\\
&=\frac{\Gamma(r)\Gamma(1-r)}{\Gamma(-\mu)\Gamma(1-r\alpha)}\sum_{n=0}^\infty\frac{(\lambda)_n}{n!}
B(n+r+1,-\mu+r)z^{n-\mu}\\
&=z^{-\mu}\frac{\Gamma(r)\Gamma(1-r)}{\Gamma(-\mu)\Gamma(1-r\alpha)}\sum_{n=0}^\infty
B(n+r+1,-\mu+r)\frac{(\lambda)_nz^{n}}{n!}.
\end{align*}
In view of \eqref{rel1}, we obtain the required result.
\end{proof}
%%%%%%%%%%%%%%%%%%%%%%%%%%%%%%%%%%%%%%%%%%%%%%%%%%%%%%%%%%%%%%%%%%%%%%%%
\section{Generating relations}
%%%%%%%%%%%%%%%%%%%%%%%%%%%%%%%%%%%%%%%%%%%%%%%%%%%%%%%%%%%%%%%%%%%%%%%
In this section, we derive generating relations of linear and bilinear type for the extended hypergeometric functions.
\begin{theorem}
The following generating relation holds true:
\begin{align}\label{fd4}
\sum_{n=0}^\infty\frac{(\lambda)_n}{n!}\,_2F_{1;p}^{\alpha}\Big(\lambda+n,\beta;\gamma;z\Big)t^n
=(1-t)^{-\lambda}\,_2F_{1;p}^{\alpha}\Big(\lambda,\beta;\gamma;\frac{z}{1-t}\Big),
\end{align}
where $|z|<\min(1, |1-t|)$, $\Re(\lambda)>0$, $\Re(\alpha)>0$, $\Re(\gamma)>\Re(\beta)>0$.

\end{theorem}
\begin{proof}
Consider the following series identity
\begin{eqnarray*}
[(1-z)-t]^{-\lambda}=(1-t)^{-\lambda}[1-\frac{z}{1-t}]^{-\lambda}.
\end{eqnarray*}
Thus, the power series expansion yields
\begin{eqnarray}\label{fd5}
\sum_{n=0}^\infty\frac{(\lambda)_n}{n!}(1-z)^{-\lambda}\Big(\frac{t}{1-z}\Big)^n=(1-t)^{-\lambda}[1-\frac{z}{1-t}]^{-\lambda}.
\end{eqnarray}
Multiplying both sides of (\ref{fd5}) by $z^{\beta-1}$ and then applying the operator $\mathfrak{D}_{z;p}^{\beta-\gamma;\alpha}$ on both sides, we have
\begin{align*}
\mathfrak{D}_{z;p}^{\beta-\gamma;\alpha}\Big[\sum_{n=0}^\infty\frac{(\lambda)_n}{n!}
(1-z)^{-\lambda}\Big(\frac{t}{1-z}\Big)^nz^{\beta-1}\Big]=(1-t)^{-\lambda}\mathfrak{D}_{z;p}^{\beta-\gamma;\alpha}
\Big[z^{\beta-1}\Big(1-\frac{z}{1-t}\Big)^{-\lambda}\Big].
\end{align*}
Interchanging the order of summation and the operator $\mathfrak{D}_{z;p}^{\beta-\gamma;\alpha}$, we have
\begin{align*}
\sum_{n=0}^\infty\frac{(\lambda)_n}{n!}\mathfrak{D}_{z;p}^{\beta-\gamma;\alpha}\Big[
z^{\beta-1}(1-z)^{-\lambda-n}\Big]t^n=(1-t)^{-\lambda}\mathfrak{D}_{z;p}^{\beta-\gamma;\alpha}
\Big[z^{\beta-1}\Big(1-\frac{z}{1-t}\Big)^{-\lambda}\Big].
\end{align*}
Thus by applying Theorem \ref{th3}, we obtain the required result.
\end{proof}
\begin{theorem}
The following generating relation holds true:
\begin{align}\label{fd6}
\sum_{n=0}^\infty\frac{(\beta)_n}{n!}\,_2F_{1;p}^{\alpha}\Big(\delta-n,\beta;\gamma;z\Big)t^n
=(1-t)^{-\beta}F_{1}\Big(\lambda,\delta,\beta;\gamma;-\frac{zt}{1-t};p,\alpha\Big),
\end{align}
where $|z|<\frac{1}{1+|t|}$, $\Re(\delta)>0$, $\Re(\beta)>0$, $\Re(\gamma)>\Re(\lambda)>0$.

\end{theorem}
\begin{proof}
Consider the series identity
\begin{eqnarray*}
[1-(1-z)t]^{-\beta}=(1-t)^{-\beta}\Big[1+\frac{zt}{1-t}\Big]^{-\beta}.
\end{eqnarray*}
Using the power series expansion to the left sides, we have
\begin{eqnarray}\label{fd7}
\sum_{n=0}^\infty\frac{(\beta)_n}{n!}(1-z)^nt^n=(1-t)^{-\beta}\Big[1-\frac{-zt}{1-t}\Big]^{-\beta}.
\end{eqnarray}
Multiplying both sides of (\ref{fd7}) by $z^{\alpha-1}(1-z)^{-\delta}$ and applying the operator $\mathfrak{D}_{z;p}^{\lambda-\gamma;\alpha}$ on both sides, we have
\begin{eqnarray*}
\mathfrak{D}_{z;p}^{\lambda-\gamma;\alpha}\Big[\sum_{n=0}^\infty\frac{(\beta)_n}{n!}z^{\alpha-1}(1-z)^{-\delta+n}t^n\Big]
=(1-t)^{-\beta}\mathfrak{D}_{z;p}^{\lambda-\gamma;\alpha}\Big[z^{\lambda-1}(1-z)^{-\delta}\Big(1-\frac{-zt}{1-t}\Big)^{-\beta}
\Big],
\end{eqnarray*}
where $\Re(\lambda)>0$ and $|zt|<|1-t|$,
thus by Theorem \ref{th2}, we have
\begin{eqnarray*}
\sum_{n=0}^\infty\frac{(\beta)_n}{n!}\mathfrak{D}_{z;p}^{\lambda-\gamma;\alpha}\Big[z^{\lambda-1}(1-z)^{-\delta+n}\Big]t^n
=(1-t)^{-\beta}\mathfrak{D}_{z;p}^{\lambda-\gamma;\alpha}\Big[z^{\lambda-1}(1-z)^{-\delta}\Big(1-\frac{-zt}{1-t}\Big)^{-\beta}
\Big].
\end{eqnarray*}
Applying Theorem \ref{th4} on both sides, we get the required result.
\end{proof}

\begin{theorem}\label{th7}
The following result holds true:
\begin{eqnarray}\label{fd8}
\mathfrak{D}_{z;p}^{\eta-\mu;\alpha}\Big[z^{\eta-1}E_{\gamma,\delta}^{\mu}(z)\Big]=\frac{z^{\mu-1}}{\Gamma(\mu-\eta)}
\sum_{n=0}^\infty\frac{(\mu)_n}{\Gamma(\gamma n+\delta)}B_{p}^{\alpha}(\eta+n,\mu-\eta)\frac{z^n}{n!},
\end{eqnarray}
where $\gamma,\delta,\mu, \alpha\in\mathbb{C}$, $\Re(p)>0$,  $\Re(\mu)>\Re(\eta)>0$ and $E_{\gamma,\delta}^{\mu}(z)$ is Mittag-Leffler function (see \cite{Prabhakar}) defined as:
\begin{eqnarray}\label{fd9}
E_{\gamma,\delta}^{\mu}(z)=\sum_{n=0}^\infty\frac{(\mu)_n}{\Gamma(\gamma n+\delta)}\frac{z^n}{n!}, \Re(\gamma)>0.
\end{eqnarray}
\begin{proof}
Using (\ref{fd9}) in (\ref{fd8}), we have
\begin{eqnarray*}
\mathfrak{D}_{z;p}^{\eta-\mu;\alpha}\Big[z^{\eta-1}E_{\gamma,\delta}^{\mu}(z)\Big]=
\mathfrak{D}_{z;p}^{\eta-\mu;\alpha}\Big[z^{\eta-1}\Big\{\sum_{n=0}^\infty\frac{(\mu)_n}{\Gamma(\gamma n+\delta)}\frac{z^n}{n!}\Big\}\Big].
\end{eqnarray*}
By Theorem \ref{th2}, we have
\begin{eqnarray*}
\mathfrak{D}_{z;p}^{\eta-\mu;\alpha}\Big[z^{\eta-1}E_{\gamma,\delta}^{\mu}(z)\Big]=
\sum_{n=0}^\infty\frac{(\mu)_n}{\Gamma(\gamma n+\delta)}\Big\{\mathfrak{D}_{z;p}^{\eta-\mu;\alpha}\Big[z^{\eta+n-1}\Big]\Big\}.
\end{eqnarray*}
Applying Theorem \ref{th1}, we get the required proof.
\end{proof}
\end{theorem}
\begin{remark}
In \cite{Mehrez} Mehrez and Tomovski introduced a new $p$-Mittag-Leffler function defined by
\begin{align}\label{MT}
E_{\lambda,\gamma,\delta;p}^{(\mu,\eta,\omega)}(z)=\sum_{k=0}^\infty\frac{(\mu)_k}{[\Gamma(\gamma k+\delta)]^\lambda}
\frac{B_p(\eta+k,\omega-\eta)}{B(\eta,\omega-\eta)}\frac{z^k}{k!}, (z \in \mathbb{C})
\end{align}
$$(\mu,\,\eta,\,\omega,\,\gamma,\,\delta,\,\lambda>0,\, \Re(p)>0).$$

\end{remark}
Hence we get the following corollary.
\begin{corollary} The following formula holds true for $E_{\gamma,\delta}^{\mu}(z)$:
\begin{align}
\mathfrak{D}_{z}^{\eta-\mu}\Big\{z^{\eta-1}E_{\gamma,\delta}^{\mu}(z);p\Big\}=
z^{\mu-1}\frac{\Gamma(\eta)}{\Gamma(\mu)}E_{1,\gamma,\delta;p}^{(\mu,\eta,\mu)}(z),
\end{align}
$$(\mu,\eta,\,\gamma,\,\delta>0,\, \Re(p)>0).$$
\end{corollary}

\begin{theorem}
The following result holds true:
\begin{align}\label{fd10}
\mathfrak{D}_{z;p}^{\eta-\mu;\alpha}\Big\{z^{\eta-1}\,_m\Psi_n\left[
                                                                        \begin{array}{cc}
                                                                          (\alpha_i,A_i)_{1,m}; &  \\
                                                                           & |z \\
                                                                          (\beta_j,B_j)_{1,n}; &  \\
                                                                        \end{array}
                                                                      \right]
\Big\}&=\frac{z^{\mu-1}}{\Gamma(\mu-\eta)}\notag\\
&\times\sum_{k=0}^\infty\frac{\prod_{i=1}^{m}\Gamma(\alpha_i+A_ik)}{\prod_{j=1}^{n}\Gamma(\beta_j+B_jk)}
B_{p}^{\alpha}(\eta+k,\mu-\eta)\frac{z^k}{k!},
\end{align}
where  $\Re(p)>0$, $\Re(\alpha)>0$, $\Re(\mu)>\Re(\eta)>0$ and $_m\Psi_n(z)$ represents the Fox-Wright function (see \cite{Kilbas}, pp. 56-58)
\begin{eqnarray}\label{fd11}
_m\Psi_n(z)=\,_m\Psi_n\left[
 \begin{array}{cc}
 (\alpha_i,A_i)_{1,m}; &  \\
 & |z \\
 (\beta_j,B_j)_{1,n}; &  \\
\end{array}
\right]=\sum_{k=0}^\infty\frac{\prod_{i=1}^{m}\Gamma(\alpha_i+A_ik)}{\prod_{j=1}^{n}\Gamma(\beta_j+B_jk)}\frac{z^k}{k!}.
\end{eqnarray}
\end{theorem}
\begin{proof}
Applying Theorem \ref{th1} and followed the same procedure used in Theorem \ref{th7}, we get the desired result.
\end{proof}
\begin{remark}
In \cite{Sharma} Sharma and Devi introduced extended Wright generalized
hypergeometric function defined by
\begin{align}\label{fd12}
_{m+1}\Psi_{n+1}(z;p)&=\,_{m+1}\Psi_{n+1}\left[
 \begin{array}{cc}
 (\alpha_i,A_i)_{1,m},(\gamma,1); &  \\
 & |(z;p) \\
 (\beta_j,B_j)_{1,n}; (c,1) &  \\
\end{array}
\right]\notag\\=&\frac{1}{\Gamma(c-\gamma)}\sum_{k=0}^\infty\frac{\prod_{i=1}^{m}\Gamma(\alpha_i+A_ik)}{\prod_{j=1}^{n}
\Gamma(\beta_j+B_jk)}\frac{B_p(\gamma+k,c-\gamma)z^k}{k!},
\end{align}
$$(\Re(p)>0,\, \Re(c)>\Re(\gamma)>0).$$
\end{remark}
Hence we get the following corollary.
\begin{corollary} The following results holds true for the Fox-Wright
hypergeometric function:
\begin{align}
\mathfrak{D}_{z}^{\eta-\mu}\Big\{z^{\eta-1}\,_{m}\Psi_{n}\left[
                                                                        \begin{array}{cc}
                                                                          (\alpha_i,A_i)_{1,m}; &  \\
                                                                           & |z \\
                                                                          (\beta_j,B_j)_{1,n}; &  \\
                                                                        \end{array}
                                                                      \right]
;p\Big\}&=z^{\mu-1}{}_{m+1}\Psi_{n+1}\left[
 \begin{array}{cc}
 (\alpha_i,A_i)_{1,m},(\eta,1); &  \\
 & |(z;p) \\
 (\beta_j,B_j)_{1,n}; (\mu,1) &  \\
\end{array}
\right],
\end{align}
$$(\Re(p)>0,\, \Re(\mu)>\Re(\eta)>0).$$
\end{corollary}
\section{Concluding remarks}
In this paper, we  established a modified extension of Riemnn-Liouville fractional derivative operator.  We conclude that when $\alpha=1$, then all the results established in this paper will reduce to the  results associated with classical Riemnn-Liouville derivative operator (see \cite{Ozerslan}).
Similarly, if we letting $\alpha=1$ and $p=0$ then all the results established in this paper will reduce to the  results associated with classical Riemnn-Liouville fractional derivative operator (see \cite{Kilbas}).

\begin{remark}
The preprint of this paper is available at 'https://arxiv.org/abs/1801.05001'.
\end{remark}

\textbf{Acknowledgments}\\
The authors would like to express profound gratitude to Prof. Virginia Kiryakova for valuable comments and remarks which improved the final version of this paper\\

%\textbf{Conflict of Interests}\\
%The author(s) declare(s) that there is no conflict of interests regarding the
%publication of this article.


\vskip 20pt
 \begin{thebibliography}{0}
% \bibitem[Names(Year)]{label} or \bibitem[Names(Year)Long names]{label}.
% (\harvarditem{Name}{Year}{label} is also supported.)
% Text of bibliographic item
\bibitem{Baleanu} D. Baleanu, P. Agarwal, R. K. Parmar, M. M. Al. qurashi, S. Salahshour, \emph{Extension of the fractional derivative operator of the Riemann-Liouville}, J. Nonlinear Sci. Appl., 10(2017), 2914-2924.

\bibitem{Chaudhry1997} M. A. Chaudhry,  A.  Qadir, M.  Rafique, S. M.  Zubair,\emph{Extension of Euler’s beta function}, J. Comput. Appl. Math. 78 (1997) 19-32.

\bibitem{Chaudhrya2004} M. A. Chaudhry, A. Qadir, H. M. Srivastava and R. B. Paris, \emph{Extended Hypergeometric
and Confluent Hypergeometric functions}, Appl. Math. Comput., 159 (2004) 589-602.

\bibitem{Choi2014} J. Choi, A. K. Rathie, R. K. Parmar, \emph{Extension of extended beta, hypergeometric and confluent hypergeometric functions}, Honam Mathematical J. 36 (2014), No. 2, pp. 357-385.

\bibitem{Kilbas} A. A. Kilbas, H. M. Sarivastava, J. J. Trujillo, \emph{Theory and application of fractional differential equation}, North-Holland Mathematics Studies, Elsevier Sciences B.V., Amsterdam, (2006).

\bibitem{Kir1} V. Kiryakova, \emph{The multi-index Mittag-Leffler functions as an important class of special functions of fractional calculus},
Computers \& Mathematics with Applications,  59(2010), 1885--1895.

\bibitem{Kiymaz} I. O. Kiymaz, A. Cetinkaya, P. Agarwal, \emph{An extension of Caputo fractional derivative operator and iyts application}, J. Nonlinear Sci. Appl., 9 (2016), 3611-3621.

\bibitem{Luo}M. J. Luo, G. V. Milovanovic, P. Agarwal, \emph{Some results on the extended beta and extended hypergeometric functions}, Appl. Math. Comput., 248 (2014), 631-651.

\bibitem{Mehrez} K.Mehrez, Z.Tomovski, \emph{On a new (p,q)-Mathieu type power series and its applications}, to appear in Appl. Anal. Discr. Math.,  13 (2019).

\bibitem{Mubeen2017} S. Mubeen, G. Rahman, K. S. Nisar, J. Choi. M. Arshad, \emph{An extended beta function and its properties}, Far East Journal of Mathematical Sciences, 102(2017),  1545-1557.

\bibitem{Olver}F. W. J. Olver, D. W. Lozier, R. F. Boisvert, C. W. Clark (eds.), \emph{NIST handbook of mathematical functions}, With 1 CD-ROM (Windows, Macintosh and UNIX), U.S. Department of Commerce, National Institute of Standards and
Technology, Washington, DC; Cambridge University Press, Cambridge, (2010).

\bibitem{Ozerslan} M. A. \"{O}zerslan, E. \"{O}zergin, \emph{Some generating relations for extended hypergeometric functions via
generalized fractional derivative operator}, Mathematical and Computer Modelling, 52 (2010), 1825-1833.

\bibitem{Ozergin} ---------, \emph{Extension of gamma, beta and
hypergeometric functions}, J. Comput. Appl. Math. 235 (2011), 4601-4610.

\bibitem{Parmar}R. K. Parmar, \emph{Some generating relations for generalized extended hypergeometric functions involving generalized fractional derivative operator}, J. Concr. Appl. Math., 12 (2014), 217-228.

\bibitem{Prabhakar} T.R. Prabhakar, \emph{A singular integral equation with a generalized Mittag–Leffler function in the kernel}, Yokohama Math. J. 19, 7–15, (1971).

\bibitem{Rahman} G. Rahman, S. Mubeen, K. S. Nisar, J. Choi, {Extended special functions and fractional integral operator via an extended Beta function}, Submitted.

\bibitem{Rainville1960}E. D. Rainville, \emph{Special functions}, The Macmillan Company, New York, 1960.

\bibitem{Samko1933} S. G. Samko, A. A. Kilbas, O. I. Marichev, \emph{Fractional integrals and derivatives, Theory and applications}, Edited and
with a foreword by S. M. Niko\'{l}ski\v{ı}, Translated from the 1987 Russian original, Revised by the authors, Gordon
and Breach Science Publishers, Yverdon, (1993).

\bibitem{Choi2017} M. Shadab, S. Jabee and J. Choi, \emph{An extension of beta function and its application}, Far East Journal of Mathematical Sciences, 103(2018), No. 1, pp 235-251.

\bibitem{Sharma}S. C. Sharma, M. Devi, \emph{Certain properties of extended Wright generalized hypergeometric function }, Annals of Pure and Applied Mathematics
Vol. 9, No. 1, 2015, 45-51.

\bibitem{Srivastava2012}H. M. Srivastava, R. K. Parmar, P. Chopra, \emph{A class of extended fractional derivative operators and associated generating relations involving hypergeometric functions}, Axioms, 1 (2012), 238-258.

\end{thebibliography}
\end{document}